\documentclass{amsart}
 \usepackage{amssymb,amsmath,amsfonts,epsfig,latexsym}
 
 \newtheorem{theorem}{Theorem}[section]

\newtheorem{proposition}[theorem]{Proposition}
\newtheorem{lemma}[theorem]{Lemma}
\newtheorem{corollary}[theorem]{Corollary}

\theoremstyle{definition}

\def\P{\mathbb{P}}
\def\Q{\mathbb{Q}} 
\def\R{\mathbb{R}}

\def\Z{\mathbb{Z}}

\def\O{\mathcal{O}}

\def\<{\langle}
\def\>{\rangle}

\DeclareMathOperator{\vol}{vol}
\DeclareMathOperator{\Vol}{Vol}

\begin{document}

\title[Cayley decompositions of lattice polytopes]{Cayley decompositions of lattice polytopes and upper bounds for $h^*$-polynomials}           

\author{Christian Haase}
\address{Inst. f\"ur Mathematik, Arnimallee 3, 14195 Berlin, Germany}
\email{christian.haase@math.fu-berline.de}

\author{Benjamin Nill}
\address{Inst. f\"ur Mathematik, Arnimallee 3, 14195 Berlin, Germany}
\email{nill@math.fu-berlin.de}
    
\author{Sam Payne}
\address{Stanford University, Mathematics, Bldg. 380, 450 Serra Mall, Stanford, CA 94305}
\email{spayne@stanford.edu}

\begin{abstract}
We give an effective upper bound on the $h^*$-polynomial of a lattice polytope in terms of its degree and leading coefficient, confirming a conjecture of Batyrev.  We deduce this bound as a consequence of a strong Cayley decomposition theorem which says, roughly speaking, that any lattice polytope with a large multiple that has no interior lattice points has a nontrivial decomposition as a Cayley sum of polytopes of smaller dimension.  

Polytopes with nontrivial Cayley decompositions correspond to projectivized sums of toric line bundles, and our approach is partially inspired by classification results of Fujita and others in algebraic geometry.  In an appendix, we interpret our Cayley decomposition theorem in terms of adjunction theory for toric varieties.
\end{abstract}
 
\maketitle

\section{Introduction}

Let $P$ be a $n$-dimensional lattice polytope and let $f_P(m)$ be the number of lattice points in $mP$, for nonnegative integers $m$.  Then the Ehrhart series $F_P(t) = \sum_{m \geq 0} f_P(m) t^m$ may be expressed as a rational function
\[
F_P(t) = \frac{h^*_P(t)} {(1-t)^{n+1}},
\]
where $h^*_P(t) = h^*_0 + h^*_1 t + \cdots + h^*_d t^d$ is a polynomial of degree $d \leq n$ with positive integer coefficients. The number $d$ 
is also called the degree of $P$. Recall that $h^*_0 + \cdots + h^*_d = \Vol(P)$ is the normalized volume of $P$, which is $n!$ times the euclidean volume.  See \cite{BeckRobins07} for this and other basic facts about Ehrhart series and $h^*$-polynomials.

In \cite{Batyrev07} Batyrev conjectured that if the degree $d$ of $h^*_P$ and the leading coefficient $h^*_d$ are both fixed, then the normalized volume of $P$, and hence each coefficient $h^*_i$, is bounded.  Since $h^*_d$ is the number of interior lattice points in $(n- d + 1) P$, if the dimension is also fixed then the volume of $P$ is bounded by a result of Hensley \cite{Hensley83}.  However, these bounds grow doubly exponentially with the dimension.

\begin{theorem} \label{main}
Let $P$ be a lattice polytope of degree $d$ such that $h^*_d = k$.  Then $\Vol(P)$ is bounded above by a number that depends only on $d$ and $k$.
\end{theorem}

\noindent The essential content of Theorem \ref{main} is that the bound on $\Vol(P)$ is independent of the dimension.  See Theorem~\ref{volume bound} for an explicit bound.  It follows that, up to lattice pyramid constructions and unimodular equivalence, there are only finitely many lattice polytopes whose $h^*$-polynomials have fixed degree and leading coefficient.  Indeed, a well-known result of Lagarias and Ziegler says that there are only finitely many lattice polytopes of given volume and dimension \cite{LagariasZiegler91}, and Batyrev recently proved that if the degree and normalized volume of $P$ are fixed then $P$ is an iterated lattice pyramid over a polytope of bounded dimension \cite{Batyrev07}.

Our approach to proving Theorem \ref{main} is to classify lattice polytopes of large dimension such that the degree of $h^*_P$ is small.  Recall that $d$ is the largest nonnegative integer such that $(n- d) P$ has no interior lattice points.  For fixed $d$, we describe lattice polytopes $P$ of dimension $n \gg 0$ such that $(n - d)P$ has no interior lattice points in terms of Cayley decompositions, as follows.

Recall that if $P_0$, $P_1, \ldots, P_s$ are lattice polytopes in $\R^q$, then the Cayley sum $P_0 * \cdots * P_s$ is defined to be the convex hull of $(P_0 \times 0) \cup (P_1 \times e_1) \cdots \cup (P_s \times e_s)$ in $\R^q \times \R^s$ for the standard basis 
$e_1, \ldots, e_s$ of $\R^s$.  A Cayley decomposition of $P$ is a $\Z$-affine linear choice of coordinates $\R^n \cong \R^q \times \R^s$ identifying $P$ with the Cayley sum $P_0 *  \cdots * P_s$ for some lattice polytopes $P_0, \ldots, P_s$ in $\R^q$.

Projection to the second factor maps $P_0 * \cdots * P_s$ onto a unimodular $s$-dimensional simplex, whose $s$-fold dilation contains no interior lattice points.  Since interior lattice points project to interior lattice points, it follows that $s  (P_0 * \cdots * P_s)$ contains no interior lattice points.  In particular, if $P$ has a decomposition as a Cayley sum of lattice polytopes in $\R^q$ then the degree of $h^*_P$ is at most $q$.  Batyrev and Nill studied lattice polytopes with multiples with no interior lattice points and asked whether all polytopes of bounded degree and sufficiently high dimension must have nontrivial Cayley decompositions \cite[Question~1.13]{BatyrevNill07}.  The following theorem gives an affirmative answer.

\begin{theorem} \label{Cayley}
Let $P$ be a lattice polytope of degree $d$.  Then $P$ decomposes as a Cayley sum of lattice polytopes in $\R^q$ for some $q \leq (d^2 + 19 d - 4) / 2$.
\end{theorem}

\noindent In particular, if $n$ is greater than $(d^2 + 19 d - 4)/2$ then $P$ has a nontrivial decomposition as a Cayley sum of at least two polytopes.  In toric geometry, polytopes with nontrivial Cayley decompositions are associated to projectivized sums of nef toric line bundles and interior lattice points of dilations correspond to sections of adjoint bundles.  Theorem~\ref{Cayley} is largely inspired by Fujita's classification results for polarized varieties $(X,L)$ such that $K_X + (n-d)L$ is not nef, for $d \leq 3$, in \cite{Fujita87}; with few exceptions, the toric examples appearing in his classification look like images of projectivized sums of line bundles.  See the Appendix for an interpretation of Theorem~\ref{Cayley} in terms of adjunction theory on toric varieties.

\smallskip

This paper is organized as follows. Section 2 contains the proof of Theorem~\ref{Cayley}. In Section 3 we improve this result in 
the special case of Gorenstein polytopes using their duality. In Section 4 we prove 
an effective version of Theorem~\ref{main}. Finally, an appendix contains in Theorem~\ref{toric} the algebro-geometric version of Theorem~\ref{Cayley}.

\vspace{10 pt}

\noindent \textbf{Acknowledgments.}  Haase and Nill are members of the Research Group Lattice Polytopes, led by Haase and supported by Emmy Noether fellowship HA 4383/1 of the German Research Foundation (DFG).  Nill would like to thank Stanford University for hospitality, and the Clay Mathematics Institute for financial support.  Payne was supported by the Clay Mathematics Institute. We thank Jaron Treutlein for a careful reading of the manuscript.

\section{Existence of Cayley decompositions}

Let $P$ be a $n$-dimensional lattice polytope in $\R^{n+1}$ that is contained in the affine hyperplane $x_{n+1} = 1$.  Let $d$ be the degree of $h^*_P(t)$ and let $x$ be a lattice point in the relative interior of $(n- d +1)P$.  Fix a $n$-dimensional lattice simplex $S \subset P$ such that $x$ is contained in the cone spanned by $S$.  To prove Theorem~\ref{Cayley} we will find a face $F$ of $S$ of dimension at most $(d^2 + 19 d - 4) / 2$ such that projection along the affine span of $F$ maps $P$ onto a unimodular simplex.  Then $P$ can be decomposed as the Cayley sum of the preimages of the vertices of this unimodular simplex.

Let $v_0, \ldots, v_n$ be the vertices of $S$.  Then any point $y$ in $\R^{n+1}$ can be written uniquely as
\[
y = b_0(y) v_0 + \cdots + b_n(y) v_n,
\]
for some real numbers $b_i(y)$.  If $F$ is a face of $S$, then projection along $F$ maps $P$ onto a unimodular simplex if and only if, 
for every lattice point $y \in \Z^{n+1}$, the coefficient $b_i(y)$ is an integer for $v_i$ not in $F$, and, moreover, 
for every vertex $w$ of $P$, the coefficient $b_i(w)$ is either zero or one for $v_i$ not in $F$, and is equal to one for at most one $v_i$ not in $F$.  
We will construct such a face $F$ in four steps, by choosing a chain of faces
\[
F_1 \subset F_2 \subset F_3 \subset F
\]
such that for every lattice point $y \in \Z^{n+1}$, $b_i(y)$ is an integer if $v_i$ is not in $F_1$, and, for every vertex $w$ of $P$, 
$b_i(w)$ is a nonnegative integer if $v_i$ is not in $F_2$, and either zero or one if $v_i$ is not in $F_3$.  We will bound the dimension of $F$ by controlling the increase in dimension at each step in this chain.

\begin{proposition}
Let $F_1$ be the face of $S$ spanned by those vertices $v_i$ such that $b_i(y)$ is not an integer for some lattice point $y \in \Z^{n+1}$.  
Then the dimension of $F_1$ is at most $4 d-2$.
\end{proposition}

\begin{proof}
This follows from \cite[Theorem~1.7]{Nill07b}.
\end{proof}

Let $\sigma$ be the cone in $\R^{n+1}$ spanned by $P$.  We write $|Z|$ for the cardinality of a finite set $Z$. 

\begin{lemma}
Let $Z$ be the set of vertices $v_i$ in $S$ such that $b_i(x)$ is zero. Then
\[
|Z| \leq d.
\]
\end{lemma}

\begin{proof}
The relative interior of the cone spanned by the vertices of $S$ that are not in $Z$ lies in the interior of $\sigma$.  Therefore, the sum of these vertices is in the interior of $(n + 1- |Z|)P$.  Since $(n - d ) P$ contains no interior lattice points, it follows that $|Z|$ is less than or equal to $d$.
\end{proof}

\begin{proposition}  \label{F2}
Let $F_2$ be the face of $S$ spanned by $F_1$ together with the vertices $v_i$ such that $b_i(w)$ is negative for some vertex $w$ of $P$.  Then the dimension of $F_2$ is at most $4d - 2 + (|Z|^2 + 7|Z|)/2$.
\end{proposition}

\noindent To prove Proposition~\ref{F2}, we will use the following three lemmas. 
For a point $y \in \R^{n+1}$, we write $Z^+(y)$ for the set of vertices $v_i \in Z$ such that $b_i(y)$ is positive and $Z^-(y)$ for the set of those $v_i \in Z$ such that $b_i(y)$ is negative.  Similarly, we write $V$ for the set of vertices $v_j$ of $S$ such that $b_j(x)$ is a positive integer, and we write $V^+(y)$, $V^-(y)$ for the set of those $v_j \in V$ such that $b_j(y)$ is positive and such that $b_j(y)$ is negative, respectively.  Since $x$ is a point in the interior of $\sigma$ with minimal last coordinate, $b_i(x)$ is at most one for every $i$.  Let $F_0 \subset F_1$ be the set of vertices $v_k$ in $S$ such that $0 < b_k(x) < 1$, and let
\[
\{x \} = \sum_{v_k \in F_0} b_k(x) v_k
\]
be the ``fractional part" of $x$.  Then
\[
x = \sum_{v_i \in V} v_i + \{x \},
\]
and $\{x \}$ is a lattice point with last coordinate $n + 1 - d - |V|$.

\begin{lemma} \label{bound V-}
For any point $y$ in $\sigma$,
\[
|V^-(y)| \leq |Z^+(y)|.
\]
\end{lemma}

\noindent In particular, $|V^-(y)|$ has size at most $|Z| \leq d$.

\begin{proof}
The cone spanned by $F_0 \cup Z^+(y) \cup (V \smallsetminus V^-(y))$ meets 
the relative interior of the cone spanned by $y$ together with $F_0 \cup Z^-(y) \cup V^-(y)$, and thus meets the interior of $\sigma$, since 
the cone spanned by $F_0 \cup V$ contains $x$ in its relative interior. Therefore,
\[
p = \{x\} + \sum_{v_i \in Z^+(y)} v_i + \sum_{v_j \in (V \smallsetminus V^-(y))} v_j
\]
is a lattice point in the interior of $\sigma$. The last coordinate of $p$ is
\[
n + 1 - d + |Z^+(y)| - |V^-(y)|.
\]
Since $(n- d)P$ contains no interior lattice points, it follows that $|V^-(y)|$ is less than or equal to $|Z^+(y)|$, as required.
\end{proof}

\begin{lemma} \label{bound V+}
If $y$ is a point in the cone spanned by $k$ vertices of $P$, then 
\[
|V^+(y)| \leq |Z^-(y)| + k.
\]
\end{lemma}

\begin{proof}
Suppose $y$ is in the cone spanned by vertices $w_1, \ldots, w_k$ of $P$.  Now, the cone spanned by $w_1, \ldots, w_k$ together with $F_0 \cup Z^-(y) \cup (V \smallsetminus V^+(y))$ meets the relative interior of the cone spanned by $F_0 \cup Z^+(y) \cup V^+(y)$, 
and hence meets the interior of $\sigma$.  Therefore,
\[
w_1 + \cdots + w_k + \{x \} + \sum_{v_i \in Z^-(y)} v_i + \sum_{v_j \in (V \smallsetminus V^+(y))} v_j
\]
is a lattice point in the interior of $\sigma$, whose last coordinate is
\[
n + 1 - d + |Z^-(y)| - |V^+(y)| + k.
\]
Since $(n - d)P$ contains no interior lattice points, it follows that $|V^+(y)|$ is less than or equal to $|Z^-(y)| + k$.
\end{proof}

\begin{lemma} \label{second bound V+}
For any point $y$ in $\sigma$ there is some $y'$ in $\sigma$ such that $V^-(y')$ contains $V^-(y)$ and $|V^+(y')| \leq |Z|$.
\end{lemma}

\begin{proof}
Let $\pi$ be the projection from $\R^{n+1}$ to the subspace spanned by $V^-(y)$, taking
$a_0 v_0 + \cdots + a_n v_n$ to $\displaystyle \sum_{v_i \in V^-(y)} a_i v_i.$  By Carath\'eodory's Theorem, for some $k \leq |V^-(y)|$ there are vertices $w_1, \ldots, w_k$ of $P$ such that
\[
\pi(y) = c_1\pi(w_1) + \cdots + c_k \pi(w_k),
\]
for some positive real numbers $c_1, \ldots, c_k$.  Let $y' = c_1 w_1 + \cdots  + c_k w_k$.  By construction, $V^-(y')$ contains $V^-(y)$.  Furthermore
\[
|V^+(y')| \leq |Z^-(y')| + k,
\]
by Lemma~\ref{bound V+}.  Now $k$ is less than or equal to $|V^-(y')|$, which is at most $|Z^+(y')|$, by Lemma~\ref{bound V-}.  Therefore, $|V^+(y')|$ is bounded above by $|Z|$, as required.
\end{proof}

\begin{proof}[Proof of Proposition~\ref{F2}]
For $y \in \sigma$, let $V(y) = V^+(y) \cup V^-(y)$ be the set of vertices of $V$ appearing with nonzero coefficients in the unique expression $y = b_0 v_0 + \cdots + b_n v_n$.  We prove the proposition by choosing a ``greedy" sequence of elements $y_0, \ldots, y_k$ in $\sigma$ such that
\[
R(j) = V^-(y_j) \smallsetminus \bigcup_{i=0}^{j-1} V(y_i)
\]
is as large as possible at each step.  By Lemma~\ref{bound V-} and 
Lemma~\ref{second bound V+}, we may choose these $y_i$ such that $|V^+(y_i)| \leq |Z|$ and $|V^-(y_i)| \leq |Z|$ for each $i$.

For $i < j$,  $V^-(y_i + y_j)$ contains both $R(j)$ and also $R(i) \smallsetminus V^+(y_j)$.  Since $y_i$ was chosen over $y_i + y_j$ at the $i$-th step, it then follows from the greedy hypothesis that
\[
|R(j)| \, \leq \, |R(i) \cap V^+(y_j)|.
\]
Now the sets $R(i) \cap V^+(y_j)$ are disjoint for fixed $j$, so $|V^+(y_j)| \geq j \cdot |R(j)|$.  
In particular, since $|V^+(y_j)|$ is at most $|Z|$, it follows that $R(j)$ is empty for $j > |Z|$. 
We set $z := \max(j \,:\, R(j) \not= \emptyset) \leq |Z|$. Let
\[
R'(j) = V^+(y_j) \smallsetminus  \bigcup_{i=0}^{j-1} V(y_i).
\]
Then, for any point $y \in \sigma$, $V^-(y)$ 
is contained in the union of the sets $R(j) \cup R'(j)$, for $j = 0, \ldots, z$.  It remains to show 
\[
\sum_{j = 0}^{z} (|R(j)| + |R'(j)|) \, \leq \, (|Z|^2 + 5|Z|) / 2,
\]
since then we can take $F_2$ to be the face spanned by $F_1$, $Z$, and $\bigcup_{j = 0}^z R(j) \cup R'(j)$.

To see this inequality, note that $|R(0)|$ and $|R'(0)|$ are each at most $|Z|$.  For $0<j\leq z$, we have 
$|R'(j)| \leq |V^+(y_j)| - \sum_{i < j} |R(i) \cap V^+(y_j)|$, which is bounded above by $|Z| - j |R(j)|$.
In particular, $|R(j)| + |R'(j)|$ is at most $|Z| - j + 1$.  Therefore,
\[
\sum_{j = 0}^z (|R(j)| + |R'(j)|) \ \leq \ 2 |Z|  + \big( |Z| + (|Z| - 1) + \cdots + 1 \big),
\]
and the proposition follows.
\end{proof}

\begin{proposition}  \label{F3}
There is a face $F_3$ of $S$ that contains $F_2$ such that
\begin{enumerate}
\item For each vertex $w$ of $P$, $b_i(w)$ is nonzero for at most one $v_i \not \in F_3$.
\item The dimension of $F_3$ is at most $4d - 2 + (|Z|^2 + 11|Z|)/2$.
\end{enumerate} 
\end{proposition}

\begin{proof}
We choose a greedy sequence of vertices $w_1, \ldots, w_r$ of $S$ that are not in $F_2$ such that
\[
\tilde{R}(j) = V^+ (w_j) \smallsetminus \big(F_2 \cup \bigcup_{i =1}^{j-1} V^+(w_i) \big)
\]
 is as large as possible at each step.  By Lemma~\ref{bound V+}, $|V^+(w_1 + \cdots + w_j)| \leq |Z| + j$. Since $b_i (w_1 + \cdots + w_j)$ is positive for $v_i \in \tilde{R}(1) \sqcup \cdots \sqcup \tilde{R}(j)$, it follows that $|\tilde{R}(j)|$ is at most one for $j$ greater than $|Z|$. In particular, we may take $F_3$ to be the face spanned by $F_2$ together with the union $\tilde{R}(1) \sqcup \cdots \sqcup \tilde{R}(|Z|)$, which has size at most $2|Z|$.
\end{proof}

\begin{proof}[Proof of Theorem~\ref{Cayley}]
Let $G$ be the set of vertices $v_i$ of $S$ that are not in $F_3$ such that $b_i(w) \geq 2$ for some vertex $w$ of $P$.  
It remains to show that $|G| \leq 2d - 2|Z|$, since then we may take $F$ to be the face of $S$ spanned by $F_3$ and $G$, and projection along $F$ will map $P$ onto a unimodular simplex, as required.

Number the vertices of $S$ so that $G = \{ v_1, \ldots, v_r \}$, $|G|=r$, and choose 
vertices $w_1, \ldots, w_r$ of $P$ such that $b_i(w_i) \geq 2$.  Then let
\[
p = w_1 / b_1(w_1) + \cdots + w_r/b_r(w_r),
\]
so $p$ is a rational point in $\sigma$ with $b_i(p) = 1$ for $1 \leq i \leq r$.

Let $p' = p + \sum_{v_i \not \in Z} c_i v_i$, where
\[
c_i = \left\{ \begin{array}{ll} 1 & \mbox{ if } b_i(p) \in \Z_{\leq 0} \\
                                        1 - \{ b_i(p) \} & \mbox{ if } b_i(p) \in \Q \smallsetminus \Z \\
                                        0 & \mbox{ otherwise.}
                                        \end{array} \right.
\]
Then $p'$ is a lattice point in the interior of $\sigma$, with last coordinate
\[
1/b_1(w_1) + \cdots + 1/b_r(w_r) + \sum c_i,
\]
which is at most $r/2+n+1-r-|Z|=n+1 - (r/2) - |Z|$.  Therefore, $(r/2) + |Z| \leq d$, and hence $r \leq 2 d - 2|Z|$.

Therefore, $P$ decomposes as a Cayley sum of lattice polytopes in $\R^q$ for some $q \leq 6d -2 + (|Z|^2 + 7 |Z|)/2$.  
Since $|Z|$ is at most $d$, the theorem follows.
\end{proof}

\section{Cayley decompositions of Gorenstein polytopes}

Recall that a lattice polytope $P$ is Gorenstein if $h^*_P$ is symmetric in the sense that $h^*_i = h^*_{d-i}$ for $0 \leq i \leq d$, where $d$ is the degree of $P$.  In particular, if $P$ is Gorenstein then $h^*_d = 1$, so $(n+1-d)P$ contains a unique lattice point, where $n$ is the dimension of $P$.  Therefore, the volumes of Gorenstein polytopes of degree $d$ are bounded uniformly by Theorem~\ref{main}.  Such bounds may be of particular interest in relation to boundedness questions from toric mirror symmetry \cite{BatyrevNill08}, and these bounds can be improved by lowering the bound for the existence of Cayley decompositions.  Here we improve the quadratic bound for general lattice polytopes in Theorem~\ref{Cayley} to a linear bound for Gorenstein polytopes.

\begin{theorem}
\label{gorst}
Let $P$ be a Gorenstein polytope of degree $d$.  Then $P$ decomposes as a Cayley sum of lattice polytopes in $\R^q$ for some $q \leq 2d-1$.
\end{theorem}

\noindent The bound in Theorem~\ref{gorst} is optimal; if $S$ is a unimodular simplex of dimension $2d-1$, then $2S$ is a Gorenstein polytope of degree $d$ which has no nontrivial Cayley decomposition, since every edge of $2S$ has lattice length two.

\begin{proof}
Let $P$ be a $n$-dimensional Gorenstein polytope of degree $d$ in $\R^{n+1}$ that is contained in the affine hyperplane $x_{n+1} = 1$.  Then the polar dual of the cone over $P$ is the cone over a dual Gorenstein polytope $P^*$, also of dimension $n$ and degree $d$.  See \cite{BatyrevBorisov97} for this and other basic facts about duality for Gorenstein polytopes.

Let $x$ be the unique lattice point in the relative interior of $(n+1-d)P^*$.  Choose a $n$-dimensional lattice simplex $S \subset P^*$ such that $x$ is contained in the cone over $S$, and order the vertices $v_0, \ldots, v_n$ of $S$ so that
\[
x = v_0 + \cdots + v_s + \{x \},
\]
where the fractional part can be written $\{x \} = \sum b_i v_i$ with $0 \leq b_i < 1$.  By \cite[Corollary~3.11]{BeckRobins07}, the sum of the coefficients $b_i$ appearing in this expression for $\{x\}$ is at most $d$.   Therefore $s$ is at least $n - 2d$, and hence $x$ can be written as a sum of $n -2d + 2$ nonzero lattice points in the cone over $P^*$.  Therefore, by \cite[Proposition~2.3]{BatyrevNill08}, $P$ must decompose as a Cayley sum of $n + 2 - 2d$ lattice polytopes in $\R^{2d-1}$, and the theorem follows.
\end{proof}

\begin{corollary}
Let $P$ be a Gorenstein polytope of degree $d$.  Then
\[
\Vol(P) \ \leq \ (2d-1)^{2d-1} \cdot \big((2d-1)!\big)^{2d} \cdot 14^{(2d-1)^{\scriptstyle 2} \cdot 2^{\scriptstyle 2d}}.
\]
\end{corollary}

\begin{proof}
Similar to the proof of Theorem~\ref{volume bound}, below, except that if $P$ is Gorenstein then $q$ is less than or equal to $N=2d - 1$.
\end{proof}

\section{Upper bounds for the normalized volume} \label{bounds}

In this section we state and prove an effective version of Theorem \ref{main}.

\begin{theorem} \label{volume bound}
Let $P$ be a lattice polytope such that $h^*_P$ has degree $d$ and $h^*_d = k$. 
Then $\Vol(P)$ is bounded above by
\[
N^N \cdot (N!)^{N+1} \cdot k^N \cdot 
\min\left(\big(7(k+1)\big)^{N^{\scriptstyle 2} \cdot 2^{\scriptstyle N + 1}},\; \big(8 N \big)^{N^{\scriptstyle 2}} \, 15^{N^{\scriptstyle 2} \cdot 2^{\scriptstyle 2 N + 1}}\right),
\]
where $N = (d^2 + 19d-4)/2$.
\end{theorem}

\begin{proof}
Let $P$ be a $n$-dimensional lattice polytope in $\R^n$ such that $h^*_P$ has degree $d$ and leading coefficient $k$.  By Theorem~\ref{Cayley}, $P$ decomposes as a Cayley sum $P \cong P_0 * \cdots * P_s$ of lattice polytopes in $\R^{q}$, for some $q \leq N$.  We may choose $q$ as small as possible, so $n = q + s$.

Let $\pi$ be the projection from $\R^n$ to $\R^s$ induced by the Cayley decomposition, and let $S$ be the standard unimodular simplex in $\R^s$ that is the image of $P$.  Set $r = n + 1 - d$, so $rP$ contains exactly $k$ interior lattice points.  The interior lattice points in $rP$ are exactly the interior lattice points in $\pi^{-1}(\lambda) \cap rP$, for interior lattice points $\lambda \in rS$.  Say $\lambda = (\lambda_1, \ldots, \lambda_s)$ is an interior lattice point of $rS$ such that $\pi^{-1}(\lambda)$ contains an interior lattice point of $P$.  Then $\lambda_1, \ldots, \lambda_s$ and
\[
\lambda_0 = r - \lambda_1 - \cdots - \lambda_s
\]
are positive integers, and $\pi^{-1}(\lambda) \cap P$ is naturally identified with the Minkowski sum $\lambda_0 P_0 + \cdots + \lambda_s P_s$.  

Let $\omega_{ij}$ be the width of $P_j$ with respect to the $i$-th coordinate on $\R^q$, the difference between the maximum and the minimum of the $i$-th coordinates of points in $P_j$.  By \cite[Theorem 2]{LagariasZiegler91} 
there is a choice of coordinates on $\R^q$ such that $\lambda_0 P_0 + \cdots + \lambda_s P_s$ is contained in the standard cube $[0,C]^q$ 
with the side length $C$ being equal to $q$ times the normalized volume of $\lambda_0 P_0 + \cdots + \lambda_s P_s$. 
By the bounds given in \cite[Theorem~1]{LagariasZiegler91}, respectively in \cite{Pikhurko01}, we may choose
\[
C = q \, q! \, \min\left( k \, \big(7(k+1)\big)^{q \cdot 2^{\scriptstyle q + 1}},\; k \, \big(8 q \big)^q \, 15^{q \cdot 2^{\scriptstyle 2 q + 1}}\right).
\] 
Since widths are additive and each $\lambda_j$ is a positive integer, it follows that $\omega_{i0} + \cdots + \omega_{is} \leq C$,
for $1 \leq i \leq q$.

Now $P$ projects onto $S$, so we can express the normalized volume of $P$ as an integral
\[
\Vol(P) = n! \cdot \int_S \vol \big( \pi^{-1}(\lambda) \cap P \big) \, d \lambda,
\]
where $\vol$ is the ordinary euclidean volume.  The volume of $\pi^{-1}(\lambda) \cap P$ is bounded by the product of its coordinate widths, which is $\prod_{i=1}^s(\lambda_0 \omega_{i0} + \cdots + \lambda_s \omega_{is})$.  Expanding the product and substituting into the integral above gives
\begin{equation} \label{expanded inequality}
\Vol(P) \  \leq n! \cdot \sum_{j_1, \ldots, j_q} \bigg( \omega_{1j_1} \cdots \omega_{qj_q} \int_S \lambda_{j_1} \cdots \lambda_{j_q} \, d \lambda \bigg),
\end{equation}
where the sum is over  $(j_1, \ldots, j_q) \in \{0, \ldots, s \}^q$.  Now it follows from H\"older's inequality that the integral over $S$ of the monomial $\lambda_{j_1} \cdots \lambda_{j_q}$ is bounded above by the integral of $\lambda_1^q$, and a straightforward induction shows that
\[
\int_S \lambda_1^q \, d \lambda = q!/n!.
\]
Substituting into (\ref{expanded inequality}) then gives
\[
\Vol(P) \ \leq  \ q! \cdot \sum_{j_1, \ldots, j_q}(\omega_{1j_1} \cdots \omega_{qj_q}).
\]
The sum on the right hand side may be written as $\prod_{i=1}^q (\omega_{i0} + \cdots + \omega_{is})$, which is bounded above by $C^q$ since, for each $i$, $\omega_{i0} + \cdots + \omega_{is}$ is less than or equal to $C$.  We conclude that $\Vol(P)$ is bounded above by $q! \cdot C^q$.  Now 
the theorem follows, since $q \leq N$.
\end{proof}

\section{Appendix: Adjunction theory for toric varieties}

Roughly speaking, adjunction theory studies polarized varieties $(X,L)$, where $X$ is a smooth $n$-dimensional complex projective variety and $L$ is an ample line bundle on $X$, with special attention to the positivity properties of the adjoint bundle $K_X + t L$ for positive integers $t$.  A prototypical result is Fujita's observation, based on Mori's Cone Theorem, that $K_X + (n+1)L$ is always nef.  In other words, the degree of the restriction of $K_X + (n+1)L$ to any curve is nonnegative.  Moreover, Fujita showed that if $K_X + nL$ is not nef then $(X,L)$ is  isomorphic to $(\P^n, \O(1))$, and he classified those polarized varieties such that $K_X + (n -d)L$ is not nef for $d \leq 4$ \cite{Fujita87}.  For an overview of adjunction theory, including refinements of these results where $X$ is allowed to have mild singularities and $r$ may be a rational number, and for further references, see \cite{BeltramettiSommese95}.

In terms of adjunction theory, Theorem~\ref{Cayley} may be interpreted as follows.

\begin{theorem} \label{toric}
Let $(X,L)$ be a polarized toric variety.  Suppose $K_X + (n-d)L$ has no nonzero global sections.  Then there is a proper birational toric morphism $\pi: X' \rightarrow X$, where $X'$ is the projectivization of a sum of line bundles on a toric variety of dimension at most $(d^2 + 19d - 4)/2$ and $\pi^*L$ is isomorphic to $\O(1)$.
\end{theorem}

\begin{proof}
Suppose $(X,L)$ is a polarized toric variety and $K_X + (n-d)L$ has no nonzero sections.  Then $L$ corresponds to an $n$-dimensional lattice polytope $P$ such that $h^*_P$ has degree at most $d$.  By Theorem~\ref{Cayley}, $P$ has a Cayley decomposition
\[
P \cong P_0 * \cdots * P_s,
\]
for some lattice polytopes $P_0, \ldots, P_s$ in $\R^q$, with $q \leq (d^2 + 19d-4)/2$.  Let $Y$ be the toric variety associated to the Minkowski sum $P_0 + \cdots + P_s$, and let $L_i$ be the line bundle on $Y$ corresponding to $P_i$.  Then $P$ is the polytope associated to $\O(1)$ on the toric variety
\[
X' \cong \P_Y(L_0 \oplus \cdots \oplus L_s).
\]
It follows that there is a proper birational toric morphism $\pi: X' \rightarrow X$ with $\pi^*L \cong \O(1)$, as required.
\end{proof}

We hope that the methods in this paper may lead to solutions of the toric cases of open problems in adjunction theory, such as \cite[Conjecture~7.1.8]{BeltramettiSommese95}, and that results such as Theorem~\ref{toric} may help shed light on what can be expected for adjunction theory of higher dimensional varieties in general.

\bibliography{math}
\bibliographystyle{amsplain}

\end{document}